\documentclass[oneside,final]{siamltex}
\pdfoutput=1
\usepackage{amsmath,amssymb,amsfonts,stmaryrd,mathtools}
\usepackage[right = 3.25 cm,left = 3.25 cm]{geometry}
\mathtoolsset{showonlyrefs=true}
\usepackage{graphicx, tikz, pgfplots}
\usepackage{subfig}
\usepackage{hhline}
\usepackage{multirow,tabularx,booktabs}
\usepackage{appendix}
\usepackage{placeins}
\usepackage{url}
\usepackage{fixltx2e}
\MakeRobust{\eqref}
\usepackage{enumitem}
\reversemarginpar
\usepackage{algorithm,algorithmic}
\usepackage{enumerate}
\usepackage{array}
\usepackage{bigstrut}
\include{MACRO}

\newtheorem{remark}[theorem]{Remark}
\newtheorem{assumption}[theorem]{Assumption}

\begin{document}
\title{A uniform additive Schwarz preconditioner for the $hp$-version of Discontinuous Galerkin approximations of elliptic problems}

\author{Paola F. Antonietti\footnotemark[2]\ \footnotemark[3] \and Marco Sarti\footnotemark[2]\ \footnotemark[4] \and Marco Verani\footnotemark[2]\ \footnotemark[5] \and Ludmil T. Zikatanov\footnotemark[6]}

\maketitle
\renewcommand{\thefootnote}{\fnsymbol{footnote}}
\footnotetext[2]{MOX-Laboratory for Modeling and Scientific Computing, Dipartimento di Matematica, Politecnico di Milano, Piazza Leondardo da Vinci 32, 20133 Milano, Italy.}
\footnotetext[3]{Email: paola.antonietti@polimi.it}
\footnotetext[4]{Email: marco.sarti@polimi.it}
\footnotetext[5]{Email: marco.verani@polimi.it}
\footnotetext[6]{Department of Mathematics, The Pennsylvania State University, University Park, PA 16802, U.S.A. and Institute of Mathematics and
Informatics, Bulgarian Academy of Sciences, Acad. G. Bonchev str, bl.~8, 1113 Sofia, Bulgaria. Email: ludmil@psu.edu}
\renewcommand{\thefootnote}{\arabic{footnote}}

\begin{abstract}
% Body of abstract:
In this paper we design
and analyze a uniform preconditioner for a class of high order
Discontinuous Galerkin schemes. The preconditioner is based
on a space splitting involving the high order conforming subspace and
results from the interpretation of the problem as a nearly-singular
problem. We show that the proposed preconditioner exhibits spectral
bounds that are uniform with respect to the discretization parameters,
\emph{i.e.}, the mesh size, the polynomial degree and the penalization
coefficient. The theoretical estimates obtained are supported by
several numerical simulations.
\end{abstract}
% Keywords:
\begin{keywords} 
discontinuous Galerkin method, high order discretizations, uniform preconditioning.
\end{keywords}

\section{Introduction}
In the last years, the design of efficient solution techniques for the
system of equations arising from Discontinuous Galerkin (DG)
discretizations of elliptic partial differential equations has become
an increasingly active field of research. On the one hand, DG methods
are characterized by a great versatility in treating a variety of
problems and handling, for instance, non-conforming grids and
$hp$-adaptive strategies. On the other hand, the main drawback of DG
methods is the larger number of degrees of freedom compared to
(standard) conforming discretizations. In this respect, the case of
high order DG schemes is particularly representative, since the
corresponding linear system of equations is very ill-conditioned: it
can be proved that, for elliptic problems, the spectral condition
number of the resulting stiffness matrix grows like $h^{-2}p^4$, $h$
and $p$ being the granularity of the underlying mesh and the
polynomial approximation degree, respectively, cf. \cite{AntHou}. As a
consequence, the design of effective tools for the solution of the
linear system of equations arising from high order DG discretizations
becomes particularly challenging. 

In the context of elliptic problems, Schwarz methods for low order DG
schemes have been studied in \cite{FengKarak01}, where overlapping and
non-overlapping domain decomposition preconditioners are considered,
and bounds of $O(H/\delta)$ and $O(H/h)$, respectively, are obtained
for the condition number of the preconditioned operator. Here $H$, $h$
and $\delta$ stand for the granularity of the coarse and fine grids
and the size of the overlap, respectively. Further extensions
including inexact local solvers, and the extension of two-level
Schwarz methods to advection-diffusion and fourth-order problems can
be found in
\cite{LassTos03,FengKarak05,AntoAyu07,AntoAyu08,AntoAyu09,DrySark10,BarkBrennSung11,AntAyuBrenSung12}. 
In the field of Balancing Domain Decomposition (BDD) methods, a number
of results exist in literature: exploiting a Neumann-Neumann type
method, in \cite{DryGalSark07,DryGalSark08} a conforming
discretization is used on each subdomain combined with interior
penalty method on non-conforming boundaries, thus obtaining a bound
for the condition number of the resulting preconditioner of $O((1-\log(H/h))^2)$. In \cite{DioDarm12}, using the unified framework
of~\cite{Arn} a BDDC method is designed and analyzed for a wide range
of DG methods. The \emph{auxiliary space method} (ASM) (see \emph{e.g.},  \cite{Nepo92,GrieOsw95,Xu96,HipXu07}) is employed in the context of $h$-version DG methods to develop, for instance,  the two-level
preconditioners of \cite{Dobretal06} and the multilevel method of
\cite{BriCamPinDahm08}. In both cases a stable splitting for the linear DG
space is provided by a decomposition consisting of a conforming
subspace and a correction, thus obtaining uniformly bounded
preconditioners with respect to the mesh size.
 
All the previous results focus on low order (\emph{i.e.},
linear) DG methods. In the context of preconditioning high order DG
methods we mention \cite{AntHou}, where a class of non-overlapping
Schwarz preconditioners is introduced, and \cite{AntAyuBerPen12},
where a quasi-optimal (with respect to $h$ and $p$) preconditioner is
designed in the framework of substructuring methods for
$hp$-Nitsche-type discretizations. A study of a BDDC scheme in the case of
$hp$-spectral DG methods is addressed in \cite{CanPavPie12}, where the
DG framework is reduced to the conforming one via the ASM.
The ASM framework is employed also in \cite{Brixetal13}, where the high order conforming space
is employed as auxiliary subspace, and a uniform multilevel
preconditioner is designed for $hp$-DG spectral element methods in
the case of locally varying polynomial degree. To the best of our
knowledge, this preconditioner is the only uniform preconditioner
designed for high order DG discretizations. We note that, in the
framework of high order methods, the decomposition involving
a conforming subspace was already employed in the case of a-posteriori
error analysis, see for example \cite{Houetal07,BurErn07,Zhuetal11}. In this paper, we address the issue of preconditioning high order DG methods by exploiting this kind of space splitting based on 
a high order conforming space and a correction. However, in our case the space decomposition is suggested by the
interpretation of the high order DG scheme in terms of a nearly-singular problem, cf. \cite{Leeetal07}.
Even though the space decomposition is similar to that of \cite{Brixetal13}, the preconditioner and the
analysis we present differs considerably since here we employ the abstract framework of subspace correction methods
provided by \cite{XuZik02}. More precisely, we are able to show that a
simple pointwise Jacobi method paired with an overlapping additive Schwarz method
for the conforming subspace, gives uniform convergence with respect to all the discretization parameters, \emph{i.e.}, the mesh size,
the polynomial order and the penalization coefficient appearing in the DG bilinear form.

The rest of the paper is organized as follows. In Section
\ref{sec:theory}, we introduce the model problem and the corresponding
discretization through a class of symmetric DG schemes. Section
\ref{sec:GLL} is devoted to few auxiliary results regarding the
Gauss-Legendre-Lobatto nodes, whose properties are fundamental to
prove the stability of the space decomposition proposed in Section
\ref{sec:decomp}. The analysis of the preconditioner is presented in
Section~\ref{sec:prec} and the theoretical results are supported by
the numerical simulations of Section~\ref{sec:num}.

\section{Model problem and $hp$-DG discretization}
\label{sec:theory}
In this section we introduce the model problem and its discretization
through several Discontinuous Galerkin schemes, see also
\cite{Arn}.\par\medskip Throughout the paper, we will employ the
notation $x\lesssim y$ and $x\gtrsim y$ to denote the inequalities
$x\leq C y$ and $x\geq C y$, respectively, $C$ being a positive
constant independent of the discretization parameters. Moreover, $x\approx y $ will mean that there exist constants
$C_1, C_2>0$ such that $C_1 y \leq x \leq C_2y$. When needed, the constants will be written explicitly.\par\medskip

Given a convex polygonal/polyhedral domain $\Om\in \mathbb{R}^d$, $d = 2,3$, and $f\in L^2(\Om)$, we consider the following weak formulation of the Poisson problem with homogeneous Dirichlet boundary conditions: find $u\in V:= H^{1}_0(\Om)$, such that
\begin{equation}
\int_\Om \nabla u \cdot \nabla v\ dx =\int_\Omega f v\ dx \, \qquad \forall v\in V.
\label{weak}
\end{equation}
Let \mesh denote a conforming quasi-uniform partition of \Om into shape-regular elements \elem of diameter \h[\elem], and set $\h := \max_{\elem\in\mesh} \h[\elem]$. We also assume that each element $\elem\in\mesh$ results from the mapping, through an affine operator $\mathsf{F}_\elem$, of a reference element \elemref, which is the open, unit $d$-hypercube in $\mathbb{R}^d$, $d=2,3$.\par 
We denote by \faceI and \faceB the set of internal and boundary faces (for $d=2$ ``face'' means ``edge'') of \mesh, respectively, and define $\face:=\faceI\cup\faceB$. 
We associate to any $F\in\face$ a unit vector $\n_F$ orthogonal to the face itself and also denote by $\n_{F,\elem}$ the outward normal vector to $F\subset\partial\elem$ with respect to \elem. We observe that for any $F\in\faceB$, $\n_{F,\elem}=\n_F$, since $F$ belongs to a unique element. For any $F\in\faceI$, we assume $\overline{F}=\partial \overline{\elem^+}\cap\partial \overline{\elem^-}$, where
\begin{align}
\elem^+ &:= \{\elem\in\mesh: F\subset\partial \elem,\ \n_F\cdot\n_{F,\elem}>0\},\\ 
\elem^- &:= \{\elem\in\mesh: F\subset\partial \elem,\ \n_F\cdot\n_{F,\elem}<0\}.
\end{align}
For regular enough vector-valued and scalar functions $\boldsymbol{\tau}$ and $v$, we denote by $\boldsymbol{\tau}^\pm$ and $v^\pm$ the corresponding traces  taken from the interior of $\elem^\pm$, respectively, and define the {\it  jumps} and {\it  averages} across the face $F\in\faceI$ as follows
\begin{alignat*}{4}
\jump[\boldsymbol{\tau}] &:= \boldsymbol{\tau}^{+} \cdot \np + \boldsymbol{\tau}^{-} \cdot \nm,\qquad{}& \average[\boldsymbol{\tau}] &:= \frac{\boldsymbol{\tau}^{+}  + \boldsymbol{\tau}^-}{2},\\
\jump[v] &:= v^{+} \np + v^{-} \nm, &\average[v] {}& := \frac{v^+  + v^-}{2},
\end{alignat*}
For $F\in\faceB$, the previous definitions reduce to $\jump[v] := v\n_F$ and $\average[\boldsymbol{\tau}] := \boldsymbol{\tau}$.\par\medskip
We now associate to the partition \mesh, the $hp$-Discontinuous Galerkin finite element space  $V_{hp}$ defined as
\begin{equation}
\label{VDG}
V_{hp}:=\{v\in L^2(\Om):v\circ \mathsf{F}_\elem \in \mathbb{Q}^{p}(\elemref)\quad\forall \elem\in \mesh\},
\end{equation}
with $\mathbb{Q}^p$ denoting the space  of all tensor-product polynomials on \elemref of degree $p>1$ in each coordinate direction. We define the lifting operators $\mcal[R](\boldsymbol{\tau}) := \sum_{F\in\face} r_F(\boldsymbol{\tau})$ and $\mcal[L](v) := \sum_{F\in\faceI} l_F(v)$, where
\begin{alignat*}{4}
r_F:[L^2(F)]^d\rightarrow[V_{hp}]^d,&\quad\int_{\Omega} r_F(\boldsymbol{\tau})\cdot\boldsymbol{\eta}\ dx &&:=-\int_{F} \boldsymbol{\tau}\cdot\average[\boldsymbol{\eta}]\ ds &&&\quad\forall F \in \face.\\
l_F:L^2(F)\rightarrow[V_{hp}]^d,&\quad\int_{\Omega} l_F(v)\cdot\boldsymbol{\eta}\ dx &&:=-\int_{F} v\jump[\boldsymbol{\eta}]\ ds &&& \quad\forall F \in \faceI,
\end{alignat*}
for any $\boldsymbol{\eta}\in [V_{hp}]^d$.\par
We then introduce the DG finite element formulation: find $u\in V_{hp}$ such that
\begin{equation}
\Aa[][u][v]=\int_\Omega f v\ dx \quad \forall v\in V_{hp},
\label{DG}
\end{equation}
with $\Aa: V_{hp}\times V_{hp}\rightarrow \mathbb{R}$ defined as
\begin{align}
\label{bilinearh}
\Aa[][u][v] := &\sum_{\elem\in\mesh}\int_\elem \nabla u \cdot \nabla v\ dx+\sum_{\elem\in\mesh}\int_\elem \nabla u \cdot (\mcal[R](\jump[v])+\mcal[L](\boldsymbol{\beta}\cdot\jump[v]))\ dx\notag\\
&+\sum_{\elem\in\mesh}\int_\elem (\mcal[R](\jump[u])+\mcal[L](\boldsymbol{\beta}\cdot\jump[u]))\cdot\nabla v\ dx+\sum_{F\in \face}\int_{F} \sigma\jump[u]\cdot \jump[v]ds\\
&+\theta\int_\Omega (\mcal[R](\jump[u])+\mcal[L](\boldsymbol{\beta}\cdot\jump[u]))\cdot(\mcal[R](\jump[v])+\mcal[L](\boldsymbol{\beta}\cdot\jump[v]))\ dx\notag,
\end{align}
where $\theta=0$ for the SIPG method of \cite{Arn82} and $\theta = 1$ for the LDG method of \cite{CoShu}. With regard to the vector function $\boldsymbol{\beta}$, we have $\boldsymbol{\beta}=\boldsymbol{0}$ for the SIPG method, while $\boldsymbol{\beta}\in\mathbb{R}^d$ is a uniformly bounded (and possibly null) vector for the LDG method.  The penalization function $\sigma\in L^\infty(\face)$ is defined as
\begin{align}\label{eq:penalty}
&\sigma|_F := 
\alpha\frac{ p^2}{\min(h_{\elem^+},h_{\elem^{-}})},
\quad F\in \faceI,
&\sigma|_F := 
\alpha\frac{ p^2}{h_\elem}
\quad F\in \faceB,
\end{align}
being $\alpha\geq 1$ and $h_{\elem^{\pm}}$ the diameters of the neighboring elements $\elem^{\pm} \in \mesh$ sharing the face $F\in\faceI$.\par 
We endow the DG space $V_{hp}$ with the following norm
\begin{equation}
\normDG[v][2] := \sum_{\elem\in\mesh} \normL[\nabla v][\elem][2]+\sum_{F\in \face}\normL[\sigma^{1/2}\jump[v]][F][2],
\label{DGnorm}
\end{equation}
and state the following result, cf. \cite{HouSchSul,PerSchot,AntHou,StaWihl}.
\begin{lemma}
\label{lem:contcoerc}
The following results hold
\begin{alignat}{2}
\Aa[][u][v]&\lesssim \normDG[u]\normDG[v]\quad &&\forall u,v\in V_{hp}, \label{cont}\\
\Aa[][u][u]&\gtrsim \normDG[u][2] &&\forall u\in V_{hp}.\label{coerc}
\end{alignat}
For the SIPG formulation coercivity holds provided the penalization coefficient $\alpha$ is chosen large enough.
\end{lemma}
From Lemma \ref{lem:contcoerc} and using the Poincar\`e inequality for piecewise $H^1$ functions of \cite{Bren03}, the following spectral bounds hold, cf. \cite{AntHou}.
\begin{lemma}
\label{lem:eigA}
For any $u\in V_{hp}$ it holds that
\begin{equation}
\label{eigA}
\sum_{\elem\in\mesh}\normL[u][\elem][2]\lesssim \Aa[][u][u] \lesssim \sum_{\elem\in\mesh}\alpha\frac{p^4}{h_\elem^{2}}\normL[u][\elem][2].
\end{equation}
\end{lemma}

\section{Gauss-Legendre-Lobatto nodes and quadrature rule}
\label{sec:GLL}
In this section we provide some details regarding the choice of the basis functions spanning the space $V_{hp}$ and the corresponding degrees of freedom. On the reference $d$-hypercube $[-1,1]^d$, we choose the basis obtained by the tensor product of the one-dimensional Lagrange polynomials on the reference interval $[-1,1]$, based on Gauss-Legendre-Lobatto (GLL) nodes. We denote by \nodesIref (\nodesBref) the set of interior (boundary) nodes of \elemref, and define $\nodesref := \nodesIref\cup\nodesBref$. The analogous sets in the physical frame are denoted by \nodesI, \nodesB and \nodes, where any ${\xi_p}\in\nodes$ is obtained by applying the linear mapping $\mathsf{F}_\elem:\elemref\rightarrow\elem$ to the corresponding $\hat{\xi}_p\in\nodesref$. The choice of GLL points as degrees of freedom allow us to exploit the properties of the associated quadrature rule. We recall that, given $(p+1)^d$ GLL quadrature nodes $\{\hat{\xi}_p\}$ and weights $\{\hat{\mathsf{w}}_{\xi_p}\}$, we have
\begin{equation}
\sum_{\hat{\xi}_p\in\nodesref}v(\hat{\xi}_p)\hat{\mathsf{w}}_{\xi_p} = \int_{\elemref} v\ dx\qquad \forall v\in \mathbb{Q}^{2p-1}(\elemref),
\end{equation}
which implies that
\begin{equation}
\sum_{\hat{\xi}_p\in\nodesref}v(\hat{\xi}_p)^2\hat{\mathsf{w}}_{\xi_p} \neq \int_{\elemref} v^2\ dx\qquad \forall v\in \mathbb{Q}^{p}(\elemref).
\end{equation}
However, by defining, for $v\in \mathbb{Q}^{p}(\elemref)$, the following norm
\begin{equation}
\label{GLLnorm}
\|v\|_{0,p,\elemref}^2 := \sum_{{\xi_p}\in\nodesref}v(\hat{\xi}_p)^2\hat{\mathsf{w}}_{\xi_p},
\end{equation}
it can be proved that 
\begin{equation}
\label{GLLequiv}
\|v\|_{0,p,\elemref}^2\approx \normL[v][\elemref][2],
\end{equation}
cf. \cite[Section~5.3]{Canetal06}. The same result holds for the physical frame \elem, \emph{i.e.}, $\|v\|_{0,p,\elem}^2\approx \normL[v][\elem][2]$.\par
Considering the Lagrange basis $\{\phi_{\xi_p}\}$, $\xi_p\in \bigcup_{\elem\in\mesh}\nodes$, we can write any $v\in V_{hp}$ as
\begin{equation}
\label{vdecomp}
v=\sum_{\elem\in\mesh}\sum_{{\xi_p}\in\nodes} v(\xi_p)\phi_{\xi_p} = \sum_{\elem\in\mesh}\sum_{{\xi_p}\in\nodes} v^{\xi_p}, 
\end{equation}
where we note that $v^{\xi_p} = v(\xi_p)\phi_{\xi_p}$.
\begin{lemma}
\label{lem:massequiv}
For any $v\in V_{hp}$, given the decomposition \eqref{vdecomp}, the following equivalence holds
\begin{equation}
\label{massequiv}
\normL[v][\Om][2]\approx \sum_{\elem\in\mesh}\sum_{{\xi_p}\in\nodes}\normL[v^{\xi_p}][\elem][2].
\end{equation}
\end{lemma}
\begin{proof}
The proof can be restricted to the case of a single element $\elem\in\mesh$. We write $v\in V_{hp}$ as in \eqref{vdecomp}, and observe that 
$$\|v^{\xi_p}\|_{0,p,\elem}^2 = \sum_{\xi_p'\in\nodes}v^{\xi_p}(\xi_p')^2\mathsf{w}_{\xi_p}=v^{\xi_p}({\xi_p})^2\mathsf{w}_{\xi_p},$$
hence, by \eqref{GLLequiv}, 
\begin{align}
\normL[v][\elem][2] \approx& \sum_{{\xi_p}\in\nodes}v({\xi_p})^2\mathsf{w}_{\xi_p} =  \sum_{{\xi_p}\in\nodes}v^{\xi_p}({\xi_p})^2\mathsf{w}_{\xi_p}\\ 
=& \sum_{{\xi_p}\in\nodes}\|v^{\xi_p}\|_{0,p,\elem}^2 \approx  \sum_{{\xi_p}\in\nodes}\normL[v^{\xi_p}][\elem][2],
\end{align}
and the thesis follows summing over all $\elem\in\mesh$.
\end{proof}

\section{Space decomposition for $hp$-DG methods}
\label{sec:decomp}
The design of our preconditioner is based on a two-stage space
decomposition: we first split the high order DG space as
$V_{hp}=V_{hp}^B+ V_{hp}^C$, with $V_{hp}^B$ denoting a proper subspace of $V_{hp}$, to be defined later, and $V_{hp}^C$ denoting the high order
conforming subspace. As a second step, both spaces are further
decomposed to build two corresponding additive Schwarz methods in each
of the subspaces. The final preconditioner on $V_{hp}$ is then
obtained by combining the two subspace preconditioners.
The first space splitting is suggested by the interpretation of the
$hp$-DG formulation \eqref{DG} as a nearly-singular problem. To
present the motivation behind this choice, we briefly introduce the
theoretical framework of \cite{Leeetal07} regarding space
decomposition methods for this class of equations. Given a finite
dimensional Hilbert space $V$, we consider the following problem: find
$u\in V$ such that
\begin{equation}
\label{nearly}
Au=(A_0 + \epsilon A_1) u = f,
\end{equation}
where $A_0$ is symmetric and positive semi-definite and $A_1$ is symmetric and positive definite. As a consequence, if $\epsilon=0$, the problem is singular, but here we are interested in the case $\epsilon>0$ (with $\epsilon$ small), \emph{i.e.}, \eqref{nearly} is nearly-singular. In general, the conditioning of problem \eqref{nearly} degenerates for decreasing $\epsilon$, and this affects the performance of standard preconditioned iterative methods, unless proper initial guess are chosen. In the framework of space decomposition methods, in order to obtain a $\epsilon$-uniform preconditioner, a key assumption on the space splitting $V_{hp} = \sum_{i=1}^N V_i$ is needed. 
\begin{assumption}[\cite{Leeetal07}]
\label{A1}
The decomposition $V_{hp} = \sum_{i=1}^N V_i$ satisfies $$\ker(A_0) = \sum_{i=1}^N (V_i\cap\ker(A_0)),$$
where $\ker(A_0)$ is the kernel of $A_0$.
\end{assumption}\par\medskip
We now turn to our DG framework, and show that a high order DG formulation can be indeed read as a nearly-singular problem with a suitable choice of $\epsilon$. For the sake of simplicity, and without any loss of generality, we retrieve equation \eqref{nearly} working directly on a bilinear form that is spectrally equivalent to \Aa. 
To this aim, let  the bilinear forms ${\mcal[A]}_\nabla(\cdot,\cdot)$, ${\mcal[A]}_J(\cdot,\cdot)$ and $\widetilde{\mcal[A]}(\cdot,\cdot)$ be defined as
\begin{equation}
\begin{aligned}
{\mcal[A]}_\nabla(u,v) &:= \sum_{\elem\in\mesh}\int_\elem \nabla u \cdot \nabla v\ dx,\\
{\mcal[A]}_J(u,v) &:= \sum_{F\in \face}\int_{F} \jump[u]\cdot \jump[v]\ ds,\\
\widetilde{\mcal[A]}(u,v)&:={\mcal[A]}_\nabla(u,v)+\alpha\frac{p^2}{h}{\mcal[A]}_J(u,v),
\end{aligned}
\end{equation}
and let $A_\nabla$, $A_J$, and $\widetilde{A}$ be their corresponding operators. Clearly, $A_\nabla$ and $A_J$ are both symmetric and positive semi-definite, and $\widetilde{A}$ is symmetric and positive definite. Moreover, thanks to Lemma~\ref{lem:contcoerc}, and the quasi-uniformity of the partition, the following spectral equivalence result holds
\begin{equation}
\label{tildeequiv}
\Aa[][u][u]\approx \normDG[u][2]\approx \widetilde{\mcal[A]}(u,u).
\end{equation}
We can then replace formulation \eqref{DG} with the following equivalent problem
\begin{equation}
\label{problemtilde1}
\widetilde{A}u = (A_\nabla + \frac{1}{\epsilon}A_J)u = \tilde{f},
\end{equation}
with $\epsilon := h/(\alpha p^2)<1$. After some simple calculations,
we can write \eqref{problemtilde1} as
\begin{equation}
\label{problemtilde}
\left[\epsilon(A_\nabla + A_J)+(1-\epsilon)A_J\right]u = \epsilon\tilde{f},
\end{equation}
which corresponds to \eqref{nearly} with $A_1=A_\nabla + A_J$ and $A_0=(1-\epsilon)A_J$. In order to obtain a suitable space splitting satisfying Assumption~\ref{A1}, we observe that, according to the definition above, the kernel of $A_0$ is given by the space of continuous polynomial functions of degree $p$ vanishing on the boundary $\partial\Om$. We then derive the first space decomposition
\begin{equation}
\label{decompDG}
V_{hp} = V_{hp}^B + V_{hp}^C ,
\end{equation}
with
\begin{align}
V_{hp}^{B}&:=\{v\in V_{hp}:v(\xi_p)=0\quad \forall\xi_p\in \bigcup_{\elem\in\mesh}\nodesI\},\\
V_{hp}^{C}&:=\{v\in C^0(\overline{\Omega}):v\circ \mathsf{F}_\elem \in \mathbb{Q}^p(\hat\elem)\quad\forall \elem\in \mesh,\ v|_{\partial\Om}=0\}\subseteq H^1_0(\Omega),
\end{align}
\emph{i.e.}, $V_{hp}^{B}$ consists of the functions in $V_{hp}$ that
are null in any degree of freedom in the interior of any
$\elem\in\mesh$. Moreover, we observe that $V_{hp}^B\subset V_{hp}$,
and $V_{hp}^{B}\cap V_{hp}^{C}\subset V_{hp}^C$, hence Assumption~\ref{A1} is satisfied by decomposition \eqref{decompDG}, which will be the basis to develop the
analysis of our preconditioner for problem \eqref{DG}.
%%%%%%
\subsection{Technical results}
In this subsection we present several results, which will be fundamental for the forthcoming analysis. We introduce a suitable interpolation operator $\DGC:V_{hp}\rightarrow V_{hp}^C$, consisting of the Oswald operator, cf. \cite{Hopetal96,KarPa03,ElAlErn04,Bur05,BurErn07}. For any $v\in V_{hp}$, we can define on each $\elem\in\mesh$ the action of the operator \DGC, by prescribing the value of $\DGC v$ in any ${\xi_p}\in\nodes$:
\begin{equation}
\label{defQ}
\DGC v({\xi_p}) := \left\{
\begin{aligned}
&0\qquad &&\textnormal{if }{\xi_p} \in \partial\Om,\\
&\frac{1}{\textnormal{card}(\mcal[T][{\xi_p}])}\sum_{\elem\in \mcal[T][{\xi_p}]}v|_\elem({\xi_p})&& \textnormal{otherwise},
\end{aligned} \right.
\end{equation}
with $\mcal[T][{\xi_p}] := \{\elem'\in\mesh: {\xi_p}\in \elem'\}$. Note that from the above definition it follows that $v-\DGC v\in V_{hp}^B$, for any $v\in V_{hp}$.\par\medskip
In addition to the space of polynomials $\mathbb{Q}^p(\elem)$, we define $\mathbb{Q}^{p}_0(\elem)$ as
\begin{align}
\mathbb{Q}^{p}_0(\elem) &:=\{v\in \mathbb{Q}^p(\elem): v({\xi_p})=0\quad\forall{\xi_p}\in\nodesI\},
\end{align}
and state the following trace and inverse trace inequalities.
\begin{lemma}[{\cite[Lemma~3.1]{BurErn07}}]
\label{lem:trace_invtrace}
The following trace and inverse trace inequalities hold
\begin{align}
\normL[v][\partial\elem][2]&\lesssim \frac{p^2}{h_\elem}\normL[v][\elem][2]\ \forall v\in\mathbb{Q}^p(\elem),\label{trace}\\ 
\normL[v][\elem][2]&\lesssim \frac{h_\elem}{p^2}\normL[v][\partial\elem][2]\ \forall v\in\mathbb{Q}^{p}_0(\elem).\label{invtrace} 
\end{align}
\end{lemma}
\par\medskip
The next result is a keypoint for the forthcoming analysis, and can be found in \cite[Lemma~3.2]{BurErn07}; for the sake of completeness the proof is reported in the Appendix.
\begin{lemma}
\label{lem:vQv}
For any $v\in V_{hp}$, the following estimate holds
\begin{equation}
\label{ineqQ}
\normL[v-\DGC v][\elem][2]\lesssim \frac{h_\elem}{p^{2}}\sum_{F\in\faceElem}\normL[\jump[v]][F][2],
\end{equation}
with $\faceElem := \{F\in \face: F\cap \elem \neq \emptyset\}$.
\end{lemma}
Thanks to Lemma~\ref{lem:vQv} we can prove the following theorem.
\begin{theorem}
\label{thm:stab}
For any $v\in V_{hp}$, it holds that
\begin{equation}
\label{stab}
\Aa[][v-\DGC v][v-\DGC v]+\Aa[][\DGC v][\DGC v]\lesssim \Aa[][v][v],
\end{equation}
where $\DGC v\in V_{hp}^C$ is defined as in \eqref{defQ}. Then the space decomposition defined in \eqref{decompDG} is stable.
\end{theorem}
\begin{proof}
We observe that, from \eqref{eigA}, the quasi-uniformity of the mesh and Lemma~\ref{lem:vQv}, we obtain
\begin{align}
\Aa[][v-\DGC v][v-\DGC v] &\lesssim \sum_{\elem\in\mesh}\alpha\frac{p^4}{h_\elem^{2}}\normL[v-\DGC v][\elem][2]
\lesssim \alpha\sum_{\elem\in\mesh}\frac{p^4}{h_\elem^{2}}\frac{h_\elem}{p^2}\sum_{F\in\faceElem} \normL[\jump[v]][F][2]\\
&\lesssim \sum_{F\in\face}\normL[\sigma^{1/2}\jump[v]][F][2]
\lesssim \Aa[][v][v].
\label{stabvQv}
\end{align}
The upper bound \eqref{stab} follows from the triangle inequality and the above estimate
\begin{align}
\Aa[][\DGC v][\DGC v]\leq \Aa[][v-\DGC v][v-\DGC v]+\Aa[][v][v]\lesssim\Aa[][v][v].
\end{align}
For any $v\in V_{hp}$, we recall that $v-\DGC v\in V_{hp}^B$, which implies
\begin{align}
\inf_{\substack{v^B\in V_{hp}^B,v^C\in V_{hp}^C\\v^B+v^C=v}}\Aa[][v^B][v^B]+\Aa[][v^C][v^C]&\leq \Aa[][v-\DGC v][v-\DGC v]+\Aa[][\DGC v][\DGC v]\lesssim \Aa[][v][v].
\end{align}
\end{proof}

\section{Construction and analysis of the preconditioner}
\label{sec:prec}
In this section we introduce our preconditioner and analyze the condition number of the preconditioned system. Employing the nomencalture of \cite{Xu92}, the preconditioner is a parallel subspace correction method (also known as additive Schwarz preconditioner, see. \emph{e.g.}, \cite{1988LionsP-aa,1988WidlundO-aa,1990DryjaM_WidlundO-aa}). Our construction uses a decomposition in two subspaces, cf. \eqref{decompDG} below, and inexact subspace solvers. Each of the
subspace solvers is a parallel subspace correction method itself.

\subsection{Canonical representation of a parallel subspace correction method}
The main ingredients needed for the analysis of the parallel subspace correction (PSC) preconditioners are suitable space splittings and the corresponding subspace solvers (see \cite{1988LionsP-aa,1988WidlundO-aa,1990DryjaM_WidlundO-aa, Xu92,GrieOsw_additive95,XuZik02,TosWid04}). In our analysis we will use the notation and the general setting from~\cite{XuZik02}. We have the following abstract result. 
\begin{lemma}[{\cite[Lemma~2.4]{XuZik02}}]
\label{lem:sumTi} 
Let $V$ be a Hilbert space which is decomposed as
$V = \sum_{i=1}^N V_i$, $V_i\subset V$, $i=1,\dots,N$,
and $T_i:V\rightarrow V_i$, $i=1,\ldots, N$ be operators whose
restrictions on $V_i$ are symmetric and positive definite. 
For $T:=\sum_{i=1}^N T_i$ the following identity holds
\begin{equation}\label{sumTi}
\Aa[][T^{-1}v][v] = 
\inf_{\substack{v_i\in V_i \\ \sum v_i = v}} \sum_{i=1}^N\Aa[][T_i^{-1}v_i][v_i].
\end{equation}
\end{lemma}
According to the above lemma, to show a bound on the condition number of
the preconditioned system we need to show that there exist positive constants
$c$ and $C$ such that
$$
c \Aa[][v][v] \le \Aa[][T^{-1} v][v] \le C \Aa[][v][v].
$$
\begin{remark}\label{remark:Pi}
In many cases we have $T_i=P_i$, $i=1,\dots,N$, where $P_i:V\rightarrow V_i$ are the elliptic projections defined as follows: for $v\in V$, its projection $P_iv$ is the unique element
  of $V_i$ satisfying $\Aa[][P_i v][v_i] := \Aa[][v][v_i]$, for all
  $v_i\in V_i$. Note that by definition, $P_i$ is the identity on $V_i$, namely, $P_iv_i = v_i=P_i^{-1}v_i$, for all $v_i\in V_i$. Hence, for
  $T = \sum_{i=1}^N P_i$, the relation \eqref{sumTi} gives
\begin{equation}\label{sumi}
\Aa[][T^{-1}v][v] = \inf_{\substack{v_i\in V_i\\\sum v_i = v}} \sum_{i=1}^N\Aa[][v_i][v_i].
\end{equation}
\end{remark}

\subsection{Space splitting and subspace solvers}
To fix the notation, let us point out that in what follows we use $T$ (with subscript when necessary) to denote (sub)space solvers and preconditioners. Accordingly, $P$ with subscript or superscript will denote elliptic projection on the corresponding subspace, which will be clear from the context.\par\medskip
 
We now define the space splitting and the corresponding subspace solvers. We recall the space decomposition from
Section~\ref{sec:decomp}, $V_{hp} = V_{hp}^B + V_{hp}^C$, where $V_{hp}^{B}$ are all functions in $V_{hp}$ for which the degrees of freedom in the interior of any $\elem\in\mesh$ vanish, and $V_{hp}^C$ is the space of high order continuous polynomials vanishing on $\partial\Om$. Note that $V_{hp}^{B}\cap V_{hp}^{C}\neq \{0\}$, and that $V_{hp}^B$ contains non-smooth and oscillatory functions, while  $V_{hp}^C$  contains the smooth part of the space $V_{hp}$. Next, on each of these subspaces we define approximate solvers $T_B:V_{hp}\rightarrow V_{hp}^B$ and $T_C: V_{hp}\rightarrow V_{hp}^C$.\\
First, we decompose $V_{hp}^B$  as follows
\begin{equation}
\label{Jdecomp}
V_{hp}^B = \sum_{\elem\in\mesh}\sum_{{\xi_p}\in\nodesB} V^{\xi_p},
\end{equation}
where
$$
V^{\xi_p}:=\left\{v\in V_{hp}^B: v({\xi_p}')=0\textnormal{ for any }
  {\xi_p}'\in\left(\bigcup_{\elem\in\mesh}\nodesB\right)\setminus\{{\xi_p}\}\right\}.
$$
The approximate solver on $V_B$ then is a simple Jacobi method, defined as 
\begin{equation*}\label{defPJ}
T_B:V_{hp}\rightarrow V^B_{hp}, \quad 
T_B:=
\left[\sum_{\elem\in\mesh}\sum_{{\xi_p}\in\nodesB}P^{\xi_p}\right] P_B.
\end{equation*}
where $P_B$ and $P^{\xi_p}$ are the elliptic projections on $V_{hp}^B$ and $V^{\xi_p}$, respectively. 
Note that $T_B$ is defined on all of $V_{hp}$ and is also an isomorphism when restricted to $V_{hp}^B$, because the elliptic
projection $P_B$ and $P^{\xi_p}$ are the identity on $V^B_{hp}$ and $V^{\xi_p}$, respectively. In addition, the splitting
is a direct sum,  and, hence, any $v\in V_{hp}^B$ is uniquely represented as
$v = \sum_{\elem\in\mesh}\sum_{{\xi_p}\in\nodesB}v^{\xi_p}$,
$v^{\xi_p}\in V^{\xi_p}$. Then, taking $P_i=P^{\xi_p}P_B: V_{hp}\rightarrow V^{\xi_p}$, from \eqref{sumi}, we have 
\begin{equation}
\label{sumPJ}
\Aa[][T_B^{-1}v^B][v^B] =
\sum_{\elem\in\mesh}\sum_{{\xi_p}\in\nodesB}\Aa[][v^{\xi_p}][v^{\xi_p}],
\quad\forall\quad v^B\in V_{hp}^B. 
\end{equation}

Next, we introduce the preconditioner $T_C$ on $V_{hp}^C$. This is the two-level overlapping additive Schwarz method introduced in~\cite{Pava92} for high order \emph{conforming} discretizations. If
we denote by $N_V$ the number of interior vertices of \mesh, then this
preconditioner corresponds to the following decomposition of $V_{hp}^C$:
\begin{equation}
\label{decompC}
V^C_{hp} = \sum_{i=0}^{N_V} V^C_i.
\end{equation}
Here $V_0^C$ is the (coarse) space of continuous piecewise
linear functions on \mesh, and for $i=1,\ldots,N_V$, 
$V^C_i := V^C_{hp}\cap H^{1}_0(\Om_i)$, where $\Om_i$ 
is the union of the elements sharing the $i$-th
vertex (see Fig.~\ref{fig:subdomain} for a two-dimensional
example). We recall that, in the case of Neumann and mixed boundary
conditions, in order to obtain a uniform preconditioner, the
decomposition \eqref{decompC} should be enriched with the subdomains
associated to those vertices not lying on a Dirichlet boundary, see
\cite{Pava92} for details.

\begin{figure}[htb]
\centering
\includegraphics[scale=1.2]{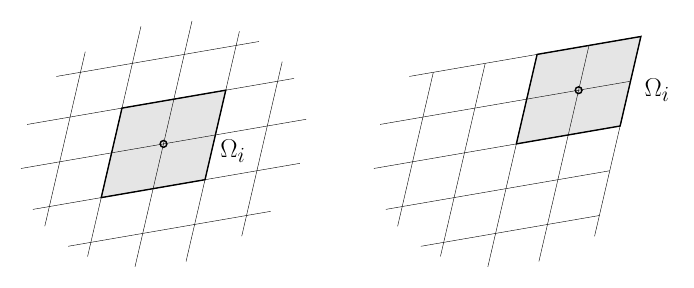}
\caption{Examples of subdomains in a two-dimensional setting.}
\label{fig:subdomain}
\end{figure}

Then, for any $V^C_i$, $i = 0,\dots,N_V$, we denote by $P_i^C:V^C_{hp}\rightarrow V^C_i$ the elliptic projections on $V^C_i$ and define the two-level overlapping additive Schwarz operator as
\begin{equation}
\label{ASM} 
T_C: V_{hp}\rightarrow V_{hp}^C, \quad 
T_C:=\left[P_0^C+\sum_{i=1}^{N_V}P_i^C\right]P_C = (P_0 + P_V)P_C,
\end{equation}
where $P_C$ is the elliptic projection on $V_{hp}^C$. As in the case of $V_B$, we have that the restriction of $T_C$ on $V_{hp}^C$ is an isomorphism.  In addition, from \eqref{sumi} with $P_i = P_i^CP_C:V_{hp}\rightarrow V_i^C$, we have 
\begin{equation}
\begin{aligned}
\label{sumPC}
\Aa[][T_C^{-1}v][v] =\inf_{\substack{v_i\in V_i^C\\\sum v_i = v}} \sum_{i=0}^{N_V}\Aa[][v_i][v_i].
\end{aligned}
\end{equation}

\subsection{Definition of the global preconditioner}
Finally, we define the global preconditioner on $V_{hp}$ by setting 
\begin{equation}
\label{PDG}
T_{DG}:V_{hp}\rightarrow V_{hp}, \quad T_{DG}:=T_{B}+T_{C},
\end{equation}
We remark that from  Lemma~\ref{lem:sumTi}, with $N=2$, $T_1=T_B$,
$V_1=V_{hp}^B$, $T_2=T_C$,
$V_2=V_{hp}^C$, we have  
\begin{align}
\label{sumPDG}
\Aa[][T_{DG}^{-1}v][v] &= \inf_{\substack{v^B\in V_{hp}^B,v^C\in V_{hp}^C\\v^B+v^C=v}} \left[\Aa[][T_B^{-1}v^B][v^B]+\Aa[][T_C^{-1}v^C][v^C]\right].
\end{align}

\subsection{Condition number estimates: subspace solvers}
We now show the estimates on the conditioning of the subspace solvers
needed to bound the condition number of $T_{DG}$. 
The first result that we prove is on the conditioning of $T_B$.
\begin{lemma}
\label{thm:jac}
Let $T_{B}$ denote the Jacobi preconditioner defined in \eqref{defPJ}.
Then there exist two positive constants $\mathsf{C}_1^J$ and
$\mathsf{C}_2^J$, independent of the granularity of the mesh $h$, the
polynomial approximation degree $p$ and the penalization coefficient $\alpha$, such that
\begin{align}
\Aa[][T_{B}^{-1}v^B][v^B]&\geq\mathsf{C}_1^J\Aa[][v^B][v^B]\quad \forall v^B\in V_{hp}^B\label{jac1}\\
\Aa[][T_{B}^{-1}(v-\DGC v)][v-\DGC v] &\leq \mathsf{C}_2^J\Aa[][v-\DGC v][v-\DGC v]\quad \forall v\in V_{hp}\label{jac2},
\end{align}
with $\DGC v$ defined in \eqref{defQ}.
\end{lemma}
\begin{proof}
  We refer to the space decomposition \eqref{Jdecomp} and write
$$v^B = \sum_{\elem\in\mesh}\sum_{{\xi_p}\in\nodesB} v^{\xi_p}.$$ 
%From \eqref{sumPJ}, it then holds that
%\begin{equation}
%\Aa[][T_{B}^{-1}v^B][v^B] = \sum_{\elem\in\mesh} \sum_{{\xi_p}\in\nodesB} \Aa[][v^{\xi_p}][v^{\xi_p}].
%\end{equation}
For the lower bound \eqref{jac1}, we employ the eigenvalue estimate \eqref{eigA} and Lemma~\ref{lem:massequiv}, thus obtaining
\begin{align}
\Aa[][v^B][v^B] \lesssim  \sum_{\elem\in\mesh}\alpha\frac{p^4}{h_\elem^2}  \normL[v^B][\elem][2]\lesssim \sum_{\elem\in\mesh}\alpha\frac{p^4}{h_\elem^2}\sum_{{\xi_p}\in\nodesB}  \normL[v^{\xi_p}][\elem][2].
\end{align}
We now observe that for any $\xi_p\in\nodesB$, $v^{\xi_p}\in \mathbb{Q}^p_0(\elem)$, and we can thus apply the inverse trace inequality \eqref{invtrace} to obtain
\begin{align}
\Aa[][v^B][v^B] \lesssim  \sum_{\elem\in\mesh}\alpha\frac{p^4}{h_\elem^2}\sum_{{\xi_p}\in\nodesB}  \normL[v^{\xi_p}][\elem][2]\lesssim \sum_{\elem\in\mesh}\alpha\frac{p^2}{h_\elem}\sum_{{\xi_p}\in\nodesB}  \normL[v^{\xi_p}][\partial\elem][2].
\end{align}
Noting that $\normL[v^{\xi_p}][\partial\elem][2]= \normL[\jump[v^{\xi_p}]][\partial\elem][2]$, it follows that
\begin{align}
\Aa[][v^B][v^B] &\lesssim \sum_{\elem\in\mesh}\alpha\frac{p^2}{h_\elem}\sum_{{\xi_p}\in\nodesB}  \normL[v^{\xi_p}][\partial\elem][2]\lesssim \sum_{\elem\in\mesh}\sum_{{\xi_p}\in\nodesB}\normL[\sigma^{1/2}\jump[v^{\xi_p}]][\partial\elem][2]\\
&\lesssim\sum_{\elem\in\mesh}\sum_{{\xi_p}\in\nodesB}\normDG[v^{\xi_p}][2],
\end{align}
and the thesis follows from the coercivity bound \eqref{coerc} and \eqref{sumPJ}.\par
With regard to the upper bound \eqref{jac2}, for the sake of simplicity we denote $w = (I-\DGC)v$, and observe that $w = (I-\DGC)w$. Since $w\in V_{hp}^B$, we write 
$$w = \sum_{\elem\in\mesh}\sum_{{\xi_p}\in\nodesB} w^{\xi_p},$$
and, from \eqref{sumPJ}, 
\begin{equation}
\Aa[][T_{B}^{-1}w][w] = \sum_{\elem\in\mesh} \sum_{{\xi_p}\in\nodesB} \Aa[][w^{\xi_p}][w^{\xi_p}].
\end{equation}
Applying again the estimate \eqref{eigA} and Lemma~\ref{lem:massequiv}, we obtain
\begin{align}
\sum_{\elem\in\mesh} \sum_{{\xi_p}\in\nodesB} \Aa[][w^{\xi_p}][w^{\xi_p}]&\lesssim \sum_{\elem\in\mesh} \sum_{{\xi_p}\in\nodesB} \alpha\frac{p^4}{h_\elem^2}\normL[w^{\xi_p}][\elem][2]\lesssim \sum_{\elem\in\mesh} \alpha\frac{p^4}{h_\elem^2}\normL[w][\elem][2]\\
&\lesssim \sum_{\elem\in\mesh} \alpha\frac{p^4}{h_\elem^2}\normL[(I-\DGC)w][\elem][2]\lesssim \Aa[][w][w],
\end{align}
where the last steps follows from Lemma~\ref{lem:vQv} and the quasi-uniformity of the mesh.
\end{proof}
\par\medskip
For the analysis of the additive preconditioner $T_C$ given
in~\eqref{ASM}, we need
several preliminary results (see \cite{Pava92} for additional details). First
of all, given the decomposition
\begin{equation}
\label{decompvC}
v = v_0+\sum_{i=1}^{N_V} v_i\qquad \forall v\in V^C_{hp},\ v_0\in V_0^C,\ v_i\in V^C_i,
\end{equation}
we define the coarse function $v_0$ as the $L^2$-projection on the space $V_0^C$, \emph{i.e.}, $v_0 := \mcal[I][0] v$ with $\mcal[I][0] v$ satisfying
\begin{align}
\normL[v-\mcal[I][0] v][\Om][2]&\lesssim h^2 |v|_{H^1(\Om)}^2,\label{approxI}\\
|\mcal[I][0] v|_{H^1(\Om)}^2&\lesssim |v|_{H^1(\Om)}^2\label{stabI},
\end{align}
for any $v\in H_0^1(\Om)$. For any $i=1,\dots,N_V$, the functions $v_i$ appearing in \eqref{decompvC} are defined as
\begin{equation}
v_i := I_p (\theta_i (v-v_0)),
\end{equation}
where $\theta_i$ is a proper partition of unity and $I_p$ is an interpolation operator, described in the following.\par
For any $\Om_i$, $i=1,\dots,N_V$, the partition of unity $\theta_i$ is such that $\theta_i\in V^C_{h1}$ and it can be defined by prescribing its values at the vertices $\{\vertex\}$ belonging to $\overline{\Om}_i$, and imposing it to be zero on $\Om\setminus\overline{\Om}_i$, see Fig.~\ref{fig:thetai} for $d=2$. More precisely,
\begin{equation}
\label{defthetai}
\theta_i(\vertex) = \left\{\begin{aligned}
1\qquad &\textnormal{ if $\vertex$ is the internal vertex or $\mcal[F][\vertex]\subset\faceB$},\\
0\qquad&\textnormal{ otherwise},
\end{aligned}\right.
\end{equation} 
with $\mcal[F][\vertex]:=\{F\in\face, F\subseteq \partial\Om_i:\vertex\in F\}$.

\begin{figure}[htb]
\centering
\includegraphics[scale=1.2]{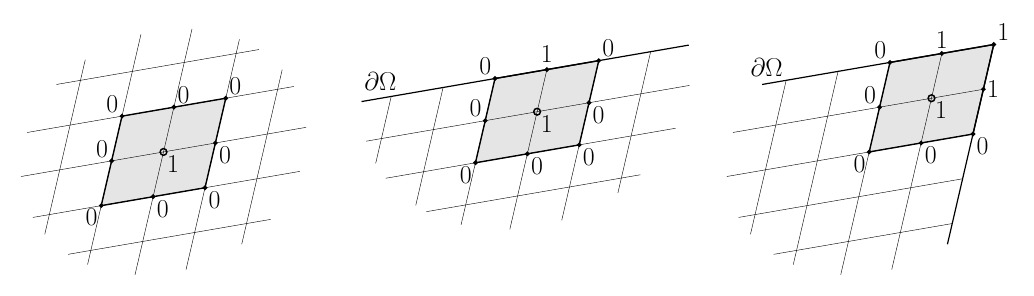}
\caption{Values of the partition of unity $\theta_i$ for $d=2$.}
\label{fig:thetai}
\end{figure}

It follows that:
\begin{equation}
\label{thetaprop}
\textnormal{supp}(\theta_i) = \Om_i,\qquad0\leq\theta_i\leq1,\qquad\sum_{i=1}^{N_V} \theta_i = 1,\qquad |\nabla\theta_i|\lesssim \frac{1}{h}.
\end{equation}
As interpolation operator $I_p$, we make use of the operator defined in \cite{Pava92}: setting $z:=v-v_0$, we define
\begin{equation}
\label{defIp}
I_p(\theta_i z)({\xi_p}) = (\theta_iz)({\xi_p})\quad\forall{\xi_p}\in\nodes,\forall\elem\in\Om_i.
\end{equation}
Notice that, despite defined locally, $I_p(\theta_i z)$ belongs to $V_{i}^C$ since the interelement continuity is guaranteed by the fact that the $(p+1)^{d-1}$ GLL points on a face uniquely determine a tensor product polynomial of degree $p$ defined on that face. The following result holds. 

\begin{lemma}[{\cite[{Lemma~3.1, Lemma~3.3}]{Pava92}}]
The interpolation operator\linebreak$I_p:\mathbb{Q}^{p+1}(\elemref)\rightarrow\mathbb{Q}^{p}(\elemref)$, defined in \eqref{defIp}, is bounded uniformly in the $H^1$ seminorm, \emph{i.e.},
\begin{equation}
\label{Ip}
|I_p(u)|_{H^1(\elemref)}\lesssim |u|_{H^1(\elemref)}\qquad\forall u\in \mathbb{Q}^{p+1}(\elemref).
\end{equation}
\end{lemma}\par
Once the partition of unity and the interpolation operator are
defined, we are able to complete the analysis of $T_C$. In analogy to
Lemma~\ref{thm:jac}, which is based on \eqref{sumPJ}, we now use
\eqref{sumPC} and the above auxiliary results to show the following lemma.
\begin{lemma}
\label{thm:PC}
Let $T_C$ denote the two-level overlapping additive Schwarz preconditioner defined in \eqref{ASM}. Then there exist two positive constants $\mathsf{C}_1^C$ and $\mathsf{C}_2^C$, independent of the discretization parameters, \emph{i.e.}, the granularity of the mesh $h$ and the polynomial approximation degree $p$, such that
\begin{align}
\Aa[][T_C^{-1}v][v]&\geq\mathsf{C}_1^C\Aa[][v][v]\label{add1}\\
\Aa[][T_C^{-1}v][v] &\leq \mathsf{C}_2^C\Aa[][v][v]\label{add2},
\end{align}
for any $v\in V_{hp}^C$.
\end{lemma}
\begin{proof}
We first prove the lower bound~\eqref{add1}, and given the decomposition \eqref{decompvC}, we can write
\begin{equation}
\Aa[][v][v] = \sum_{i,j=0}^{N_V} \Aa[][v_i][v_j]\lesssim \Aa[][v_0][v_0] + \sum_{i,j=1}^{N_V} \Aa[][v_i][v_j].
\end{equation} 
We now note that $\Aa[][v_i][v_j]\neq 0$ only if $i=j$ and $\Om_i\cap\Om_j\neq\emptyset$, and since each $\Om_i$ is overlapped by a limited number of neighboring subdomains, we conclude that 
\begin{align}
\Aa[][v][v] &\lesssim \Aa[][v_0][v_0]+ \sum_{i,j=1}^{N_V}\Aa[][v_i][v_j]\lesssim\Aa[][v_0][v_0]+ \sum_{i=1}^{N_V}\Aa[][v_i][v_i].
\end{align}
Inequality \eqref{add1} follows from the bound above and \eqref{sumPC}, denoting with $\mathsf{C}_1^C$ the hidden constant.\par 
Note that, from \eqref{sumPC}, the upper bound \eqref{add2} is proved provided the following inequality holds
\begin{equation}
\label{stablesplit}
\sum_{i=0}^{N_V} \Aa[][v_i][v_i] \leq \mathsf{C}_2^C \Aa[][v][v]\qquad\forall v\in V_{hp}^C.
\end{equation}
We recall that $v_0 =\mcal[I][0]v$, and from \eqref{stabI} it follows that
\begin{equation}
\label{A0}
\Aa[][v_0][v_0]=\Aa[][\mcal[I][0]v][\mcal[I][0]v]\lesssim \Aa[][v][v].
\end{equation}
For $i = 1,\dots, N_V$, we have $v_i = I_p(\theta_i z)$, with $z=v-v_0$, and by \eqref{Ip}, we obtain
\begin{equation}
|v_i|_{H^1(\elem')}^2\lesssim |\theta_i z|_{H^1(\elem')}^2 \lesssim \sum_{j=1}^d \left\|\frac{\partial\theta_i}{\partial x_j}z+\theta_i\frac{\partial z}{\partial x_j}\right\|_{L^2(\elem')}^2,
\end{equation}
for any $\elem'\in\Om_i$. By \eqref{thetaprop} it holds that
\begin{align}
|\nabla\theta_i|\lesssim \frac{1}{h},\qquad\|\theta_i\|_{L^\infty}\leq 1,
\end{align}
hence,
\begin{align}
|v_i|_{H^1(\elem')}^2\lesssim \frac{1}{h^2}\normL[z][\elem'][2]+\sum_{j=1}^d \normL[\frac{\partial z}{\partial x_j}][\elem'][2]\lesssim \frac{1}{h^2}\normL[v-v_0][\elem'][2]+|v-v_0|^2_{H^1(\elem')}.
\end{align}
On any element $\elem'$, a limited number of components $v_i$ are different from zero (at most four for $d=2$, and eight for $d=3$), which implies that we can sum over all the components $v_i$, $i = 1,\dots,N_V$, and then over all the elements, thus obtaining
\begin{align}
\sum_{i=1}^{N_V}|v_i|_{H^1(\Om)}^2&\lesssim \frac{1}{h^2}\normL[v-v_0][\Om][2]+|v-v_0|^2_{H^1(\Om)}\lesssim |v|^2_{H^1(\Om)},
\end{align}
where the last step follows from \eqref{approxI} and \eqref{stabI}. The addition of the above result and \eqref{A0}, gives \eqref{stablesplit}, denoting with $\mathsf{C}_2^C$ the resulting hidden constant.
\end{proof}\par

\subsection{Condition number estimates: global preconditioner}
We are now ready to prove the main result of the paper regarding the
condition number of the preconditioned problem.
\begin{theorem}
\label{thm:main}
Let $T_{DG}$ be defined as in \eqref{PDG}. Then, for any $v\in V_{hp}$, it holds that
\begin{equation}
\label{main}
\Aa[][v][v]\lesssim \Aa[][T_{DG}^{-1}v][v] \lesssim \Aa[][v][v],
\end{equation}
where the hidden constants are independent of the discretization parameters, \emph{i.e.}, the mesh size $h$, the polynomial approximation degree $p$, and the penalization coefficient $\alpha$. 
\end{theorem}
\begin{proof}
To prove the upper bound, we first consider the identity
\eqref{sumPDG}. Recalling that, by definition \eqref{defQ}, $v-\DGC v\in V_{hp}^B$, for any $v\in
V_{hp}$, we obtain
\begin{align}
\Aa[][T_{DG}^{-1}v][v] &= \inf_{\substack{v^B\in V_{hp}^B,v^C\in V_{hp}^C\\v^B+v^C=v}} \left[\Aa[][T_{B}^{-1}v^B][v^B]+\Aa[][T_C^{-1}v^C][v^C]\right]\\
&\leq \Aa[][T_{B}^{-1}(v-\DGC v)][v-\DGC v]+\Aa[][T_C^{-1}\DGC v][\DGC v].
\end{align} 
From the bounds \eqref{jac2} and \eqref{add2} for $\DGC v$, it follows that
\begin{align}
\Aa[][T_{DG}^{-1}v][v] &\leq\Aa[][T_{B}^{-1}(v-\DGC v)][v-\DGC v]+\Aa[][T_C^{-1}\DGC v][\DGC v]\\
&\lesssim\Aa[][v-\DGC v][v-\DGC v]+\Aa[][\DGC v][\DGC v]\lesssim \Aa[][v][v],
\end{align}
where the last step follows from \eqref{stab}. The lower bound follows from \eqref{sumPDG}, the bounds \eqref{jac1} and \eqref{add1}, and a triangle inequality
 \begin{align}
\Aa[][T_{DG}^{-1}v][v] &= \inf_{\substack{v^B\in V_{hp}^B,v^C\in
                         V_{hp}^C\\v^B+v^C=v}}
   \left[\Aa[][T_{B}^{-1}v^B][v^B]+\Aa[][T_C^{-1}v^C][v^C]\right]\\
&\gtrsim \inf_{v^C\in V_{hp}^C} \left[\Aa[][v^B][v^B]+\Aa[][v^C][v^C]\right]\gtrsim \Aa[][v][v].
\end{align}
\end{proof}

\section{Numerical experiments}
\label{sec:num}
In this section we present some numerical tests to verify the
theoretical estimates provided in Lemma~\ref{thm:jac},
Lemma~\ref{thm:PC} and Theorem~\ref{thm:main}. We consider problem
\eqref{DG} in the two dimensional case with $\Om=(-1,1)^2$ and SIPG
and LDG discretizations. For the first experiment, we set
$h = 0.0625$, the penalization parameter $\alpha = 10$ and
$\boldsymbol{\beta}=\boldsymbol{1}$ for the LDG method. In Table
\ref{tab:constants}, we show the numerical evaluation of the constants
$\mathsf{C}_1^J$ and $\mathsf{C}_2^J$ of Lemma~\ref{thm:jac} and
$\mathsf{C}_1^C$ and $\mathsf{C}_2^C$ of Lemma~\ref{thm:PC}, as a
function of the polynomial order employed in the discretization: the
constants are independent of $p$, as expected from theory. With regard
to the constants $\mathsf{C}_1^C$ and $\mathsf{C}_2^C$, we observe
that the values are the same for both the SIPG and LDG methods, since
the preconditioner on the conforming subspace reduces to the same
operator regardless of the DG scheme employed.

\begin{table}[H]
  \centering
  \caption{Left and middle: numerical evaluation of the constants $\mathsf{C}_1^J$ and $\mathsf{C}_2^J$ of Lemma~{\rm\ref{thm:jac}} as a function of $p$ for the SIPG and LDG methods; right: numerical evaluation of the constants $\mathsf{C}_1^C$ and $\mathsf{C}_2^C$ of Lemma~{\rm\ref{thm:PC}} as a function of $p$}
  {
\begin{tabular}{lcc||cc||cc}
\hline
\bigstrut[t] &\multicolumn{2}{c||}{SIPG ($\alpha=10$, $\boldsymbol{\beta}=\boldsymbol{0}$)}& \multicolumn{2}{c||}{LDG ($\alpha=10$, $\boldsymbol{\beta}=\boldsymbol{1}$)}& \multicolumn{2}{c}{}\\
\hline
\bigstrut[t]& $\mathsf{C}_1^J$ & $\mathsf{C}_2^J$ & $\mathsf{C}_1^J$ & $\mathsf{C}_2^J$ & $\mathsf{C}_1^C$ & $\mathsf{C}_2^C$\\
\hline
\bigstrut[t]$p = 2$ & 0.4036 & 3.0084 & 0.3844 & 3.6393 & 0.2500 & 1.1606\\
$p = 3$ & 0.4343 & 2.9133 & 0.4129 & 3.3232 & 0.2500 & 1.0742\\
$p = 4$ & 0.4502 & 2.8304 & 0.4298 & 3.1487 & 0.2500 & 1.0934\\
$p = 5$ & 0.4605 & 2.7633 & 0.4410 & 3.0321 & 0.2500 & 1.0820\\
$p = 6$ & 0.4674 & 2.7088 & 0.4489 & 2.9467 & 0.2500 & 1.0854\\
\hline
\end{tabular}
}
\label{tab:constants}
\end{table}

Table \ref{tab:condDG} shows a comparison of the spectral condition number of the original system ($\mathsf{K}(A)$) and of the preconditioned one ($\mathsf{K}(T_{DG})$). While the former grows as $p^4$, cf. \cite{AntHou}, the latter is constant with $p$, as stated in \eqref{main}. The theoretical results are further confirmed by the number of iterations $N_{iter}^{PCG}$ and $N_{iter}^{CG}$ of the Preconditioned Conjugate Gradient (PCG) and the Conjugate Gradient (CG), respectively, needed to reduce the initial relative residual of a factor of $10^{-8}$.

\begin{table}[htb!]
\centering
\caption{Condition number of the unpreconditioned ($\mathsf{K}(A)$) and preconditioned ($\mathsf{K}(T_{DG})$) linear systems of equations and corresponding CG ($N_{iter}^{CG}$) and PCG ($N_{iter}^{PCG}$) iteration counts as a function of $p$ for the SIPG and LDG methods}
%\footnotesize
{\setlength{\tabcolsep}{5pt}
\begin{tabular}{lcccc||cccc}
\hline
\bigstrut[t] &\multicolumn{4}{c||}{SIPG ($\alpha=10$, $\boldsymbol{\beta}=\boldsymbol{0}$)}&\multicolumn{4}{c}{LDG ($\alpha=10$, $\boldsymbol{\beta}=\boldsymbol{1}$)}\\
\hline
\bigstrut[t]& $\mathsf{K}(A)$ & $N_{iter}^{CG}$ & $\mathsf{K}(T_{DG})$& $N_{iter}^{PCG}$ & $\mathsf{K}(A)$ & $N_{iter}^{CG}$ & $\mathsf{K}(T_{DG})$& $N_{iter}^{PCG}$
\\
\hline
\bigstrut[t]$p=2$ & $5.26\cdot 10^3$ & 284 & 14.26 & 27 & $8.88\cdot 10^3$ & 392 & 35.02 & 36\\
$p=3$ & $1.52\cdot 10^4$ & 450 & 14.22 & 25 & $2.29\cdot 10^4$ & 556 & 38.29 & 31\\
$p=4$ & $3.38\cdot 10^4$ & 684 & 14.72 & 26 & $4.89\cdot 10^4$ & 851 & 37.74 & 33\\
$p=5$ & $6.27\cdot 10^4$ & 919 & 15.35 & 24 & $8.83\cdot 10^4$ & 1137 & 38.37 & 30\\
$p=6$ & $1.05\cdot 10^5$ & 1200 & 15.98 & 25 & $1.45\cdot 10^5$ & 1482 & 42.65 & 32\\
\hline
\end{tabular}
}
\label{tab:condDG}
\end{table}

The second numerical experiment aims at verifying the uniformity of the proposed preconditioner with respect to the penalization coefficient $\alpha$. In this case, we consider the same test case presented above, but we now fix the polynomial approximation degree $p=2$ and increase  $\alpha$. The numerical data obtained are reported in Table~\ref{tab:condDG_vs_alpha}: as done before, we compare the spectral condition numbers of the unpreconditioned and preconditioned systems and the iteration counts of the CG and PCG methods. As predicted from theory, while $\mathsf{K}(A)$ grows like $\alpha$, the values of $\mathsf{K}(T_{DG})$ are constant.

\begin{table}[htb!]
\centering
\caption{Condition number of the unpreconditioned ($\mathsf{K}(A)$) and preconditioned ($\mathsf{K}(T_{DG})$) linear systems of equations and corresponding CG ($N_{iter}^{CG}$) and PCG ($N_{iter}^{PCG}$) iteration counts as a function of $\alpha$ for the SIPG and LDG methods}
{
\setlength{\tabcolsep}{4.5pt}
\begin{tabular}{lcccc||cccc}
\hline
\bigstrut[t]&\multicolumn{4}{c||}{SIPG ($p=2$, $\boldsymbol{\beta}=\boldsymbol{0}$)}&\multicolumn{4}{c}{LDG ($p=2$, $\boldsymbol{\beta}=\boldsymbol{1}$)}\\
\hline
\bigstrut[t]& $\mathsf{K}(A)$ & $N_{iter}^{CG}$ & $\mathsf{K}(T_{DG})$& $N_{iter}^{PCG}$ & $\mathsf{K}(A)$ & $N_{iter}^{CG}$ & $\mathsf{K}(T_{DG})$& $N_{iter}^{PCG}$\\
\hline
\bigstrut[t]$\alpha=2$ & $1.04\cdot 10^3$ & 137 & 12.66 & 28 & $4.55\cdot 10^3$ & 297 & 62.54 & 47\\
$\alpha=5$ & $2.62\cdot 10^3$ & 205 & 13.02 & 28 & $6.17\cdot 10^3$ & 338 & 41.94 & 39\\
$\alpha=10$ & $5.26\cdot 10^3$ & 284 & 14.26 & 27 & $8.88\cdot 10^3$ & 392 & 35.02 & 36\\
$\alpha=10^2$ & $5.41\cdot 10^4$ & 690 & 15.73 & 28 & $5.78\cdot 10^4$ & 717 & 29.32 & 31\\
$\alpha=10^3$ & $5.44\cdot 10^5$ & 1116 & 15.90 & 28 & $5.47\cdot 10^5$ & 1142 & 28.92 & 30\\
$\alpha=10^4$ & $5.44\cdot 10^6$ & 1509 & 15.91 & 28 & $5.44\cdot 10^6$ & 1518 & 28.89 & 30\\
\hline
\end{tabular}
}
\label{tab:condDG_vs_alpha}
\end{table}

\FloatBarrier
{
\Appendix
\section{Proof of Lemma~\ref{lem:vQv}}
We first introduce some additional notation. For any $\elem\in\mesh$, we define $\partial \elem_{d-1}$ as the set of $(d-1)$-dimensional affine varieties in $\partial \elem$, and $\partial_\ell \elem$, $\ell \in \{d-2,\dots,0\}$, as the set obtained as the intersection of two distinct elements in $\partial_{\ell+1} \elem$. We observe that $\partial\elem_{d-1}\subset\face$ for any $\elem\in\mesh$. The set of nodes of each element \elem can be further decomposed as
\begin{equation}
\label{decompN}
 \nodes = \nodesI\cup\nodesB = \nodesI\cup\bigcup_{\ell = 0}^{d-1} \mathcal{V}_\ell,
\end{equation}
with \mcal[V][\ell], $\ell \in \{d-1,\dots,0\}$, representing the set of interior nodes of $\partial \elem_\ell$ (see Fig.~\ref{fig:nodes}).\\

\begin{figure}[htb]
\centering
\includegraphics[scale=0.20]{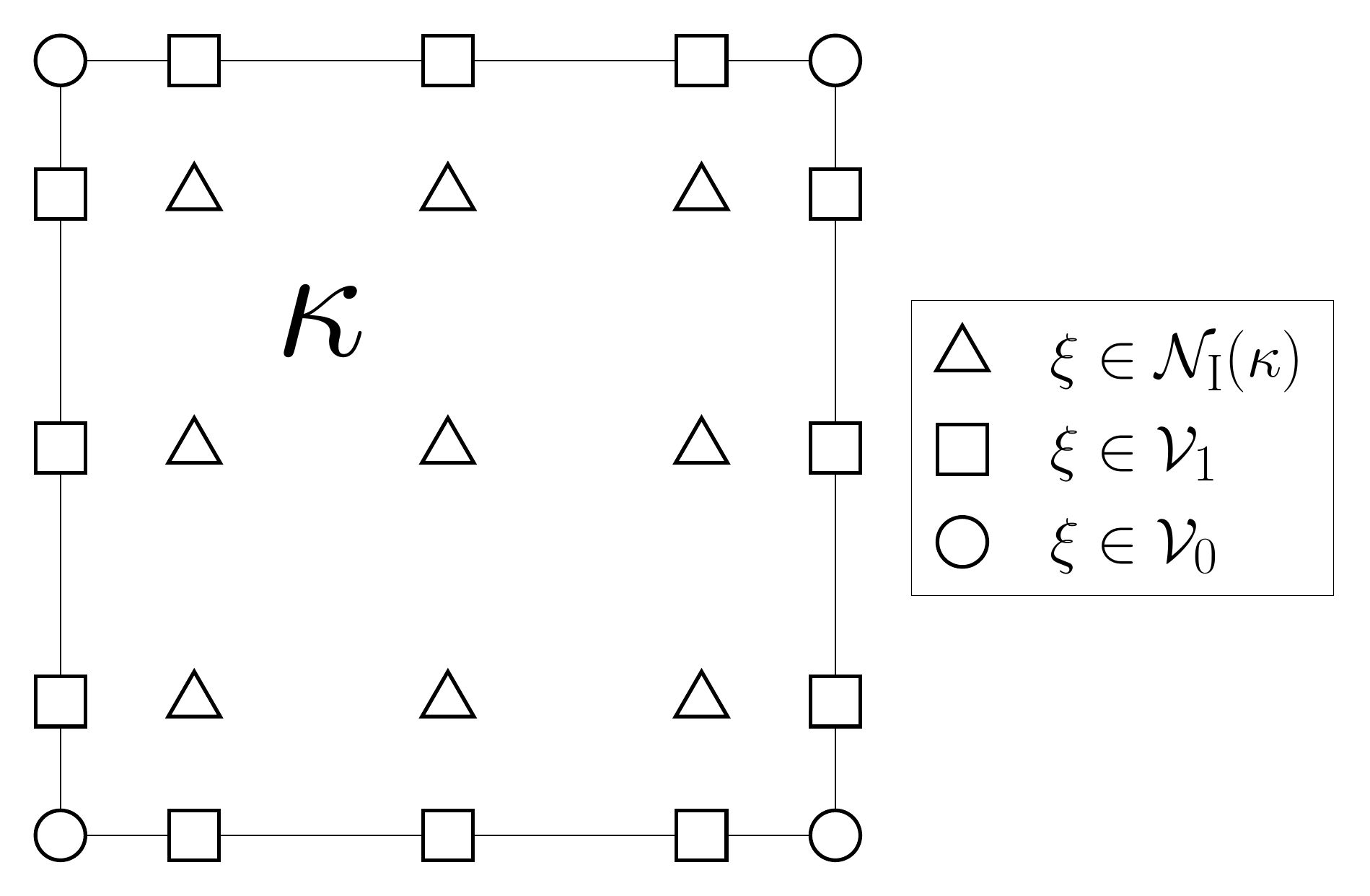}
\caption{Distribution of the nodes ${\xi_p}\in\nodes$ for $p=4$, $d=2$.}
\label{fig:nodes}
\end{figure}

In analogy to $\mathbb{Q}^{p}_0(\elem)$, we define $\mathbb{Q}^{p}_0(\partial \elem_\ell)$, $\ell \in \{d-1,\dots,1\}$, as
\begin{align}
\mathbb{Q}^{p}_0(\partial \elem_\ell) :=\{v\in \mathbb{Q}^{p}(\partial \elem_\ell): v( {\xi_p})=0\quad\forall{\xi_p}\in\mcal[V][\ell]\},
\end{align}
and remark that a corresponding form of the trace and inverse trace inequalities of Lemma~\ref{lem:trace_invtrace} can be obtained on $\mathbb{Q}^{p}(\elem_\ell)$ and $\mathbb{Q}^{p}_0(\partial \elem_\ell)$. 
\par\medskip
The proof of Lemma~\ref{lem:vQv} can be found in  \cite[Lemma~3.2]{BurErn07}. Here we reproduce the same steps, with only minor changes, mainly regarding the notation.\\
We denote $w = (v-\DGC v)|_\elem$ and observe that, according to \eqref{defQ}, it holds
$$w({\xi_p})=0\qquad \forall{\xi_p}\in\nodesI.$$
Given the set of nodes ${\xi_p}\in \nodes$ and the associated Lagrangian nodal basis functions $\{\phi_{\xi_p}\}$, we can write
$$w = \sum_{{\xi_p}\in\nodesB} w({\xi_p})\phi_{\xi_p},$$
cf. \eqref{vdecomp}. From the decomposition \eqref{decompN}, it follows that
\begin{equation}
\label{w}
w = \sum_{{\xi_p}\in\mathcal{V}_{d-1}} w({\xi_p})\phi_{\xi_p}+\sum_{\ell=0}^{d-2}\sum_{{\xi_p}\in\mcal[V][\ell]} w({\xi_p})\phi_{\xi_p}= \sum_{{\xi_p}\in\mathcal{V}_{d-1}} w({\xi_p})\phi_{\xi_p}+\sum_{\ell=0}^{d-2}\mathsf{r}_\ell,
\end{equation}
where for any $\ell\in\{0,\dots,d-2\}$ we have
\begin{align}
\mathsf{r}_\ell\in\mathbb{Q}^{p}_0(\elem),\qquad 
\mathsf{r}_\ell\in\mathbb{Q}^{p}_0(\partial \elem_l)\qquad \forall l\in\{\ell+1,\dots,d-1\}.
\end{align}
Let us introduce $\mathcal{V}_{d-1,F}$ as the set of interior nodes of $F\subset\partial \elem$. For any ${\xi_p}\in\mathcal{V}_{d-1,F}$, by \eqref{defQ}, we have that 
\begin{equation}
w({\xi_p}) = \boldsymbol{\gamma}_F\cdot\jump[v]({\xi_p}),\qquad \boldsymbol{\gamma}_F:=\left\{\begin{aligned} &\n_{F,\elem} &&\textnormal{ if } F\subset\partial\Om,\\
\frac{1}{2}&\n_{F,\elem} &&\textnormal{ otherwise}.
\end{aligned}\right.
\end{equation} 
We the above notation, we have
\begin{align}
\sum_{{\xi_p}\in\mathcal{V}_{d-1}} w({\xi_p})\phi_{\xi_p} = \sum_{F\subset \partial \elem}\psi_F,\ \textnormal{where}\quad \psi_F := \boldsymbol{\gamma}_F\cdot\sum_{{\xi_p}\in\mathcal{V}_{d-1,F}}\jump[v]({\xi_p})\phi_{\xi_p}
\end{align}
We next observe that $\psi_F({\xi_p}) = 0$ for any ${\xi_p}\in\nodesI$, \emph{i.e.}, $\psi_F\in\mathbb{Q}_0^{p}(\elem)$, and $\psi_F({\xi_p}) = 0$ also for any ${\xi_p}\in\partial \elem\setminus F$. We can then apply the inverse trace inequality \eqref{invtrace}, thus obtaining
\begin{equation}
\normL[\psi_F][\elem][2]\lesssim \frac{h_\elem}{p^2}\normL[\psi_F][\partial\elem][2] \lesssim \frac{h_\elem}{p^2}\normL[\psi_F][F].
\end{equation}
Recalling \eqref{w}, it follows
\begin{equation}
\label{wEst}
\normL[w][\elem][2]\lesssim \sum_{F\subset \partial \elem} \frac{h_\elem}{p^2}\normL[\psi_F][F][2]+\sum_{\ell=0}^{d-2}\normL[\mathsf{r}_\ell][\elem][2].
\end{equation}
In order to bound the terms on the right hand side of \eqref{wEst}, we proceed by considering $d=1,2,3$ as separate cases. For $d=1$, inequality \eqref{wEst} reduces to 
\begin{equation}
\label{wd1}
\normL[w][\elem][2]\lesssim \sum_{F\subset \partial \elem} \frac{h_\elem}{p^2}\normL[\psi_F][F][2],
\end{equation}
and we observe that by the definition of $\psi_F$ it holds that $\psi_F|_F = \boldsymbol{\gamma}_F\cdot\jump[v]$. As a consequence, \eqref{wd1} implies \eqref{ineqQ}. For $d=2$, \eqref{wEst} reduces to
\begin{equation}
\label{wd2}
\normL[w][\elem][2]\lesssim \sum_{F\subset \partial \elem} \frac{h_\elem}{p^2}\normL[\psi_F][F][2]+\normL[\mathsf{r}_0][\elem][2].
\end{equation}
First of all, we recall that the function $\psi_F$ is equal to zero on $\partial F$ and coincides with $\boldsymbol{\gamma}_F\cdot \jump[v]$ on any ${\xi_p}\in \mcal[V][d-1,F]$, which means that $\psi_F-\boldsymbol{\gamma}_F\cdot \jump[v]\in \mathbb{Q}^{p}_0(F)$. By applying the inverse trace inequality \eqref{invtrace} and the trace inequality \eqref{trace} , we get
\begin{equation}
\label{psidiff}
\normL[\psi_F-\boldsymbol{\gamma}_F\cdot \jump[v]][F][2]\lesssim \frac{h_F}{p^2}\normL[\jump[v]][\partial F][2]\lesssim\frac{h_F}{p^2}\frac{p^2}{h_F}\normL[\jump[v]][F][2]\lesssim \normL[\jump[v]][F][2].
\end{equation}
From \eqref{psidiff} and the triangle inequality, it follows 
\begin{equation}
\label{psifinal}
\normL[\psi_F][F][2]\lesssim \normL[\jump[v]][F][2].
\end{equation}
We next estimate the term \normL[\mathsf{r}_0][2]. To this aim, we recall that $\mathsf{r}_0\in\mathbb{Q}^{p}_0(\elem)$ and $\mathsf{r}_0\in\mathbb{Q}^{p}_0(F)$. This allows us to apply the inverse trace inequality \eqref{invtrace} twice, thus obtaining
\begin{align}
\normL[\mathsf{r}_0][\elem][2]\lesssim \frac{h_\elem}{p^2}\sum_{F\subset\partial \elem}\normL[\mathsf{r}_0][F][2]\lesssim \frac{h_\elem}{p^2}\sum_{F\subset\partial \elem}\frac{h_F}{p^2}\normL[\mathsf{r}_0][\partial F][2].
\end{align}
Moreover, we note that, for $d=2$, $\partial F$ is given only by two nodes and for ${\xi_p}\in\partial F$, a simple calculation leads to
\begin{equation}
\label{r0xi}
\mathsf{r}_0({\xi_p}) = \sum_{F'\in\mcal[F][{\xi_p}]} \boldsymbol{\eta}_{F,F'}({\xi_p})\cdot\jump[v]({\xi_p}),
\end{equation}
where $\mcal[F][{\xi_p}]:=\{F'\in\face: {\xi_p}\in F'\}$ and 
\begin{equation}
\boldsymbol{\eta}_{F,F'}({\xi_p}) := \left\{\begin{aligned}\pm&\frac{1}{2}\n_{F'}\quad&&\textnormal{ if ${\xi_p}\in\partial\Om$},\\
%\pm&\frac{1}{2}\n_{F',\elem}\quad&&\textnormal{ if $F'\subset \partial T\setminus\partial%\Omega$ and ${\xi_p}\in\partial\Omega$},\\
 \pm&\frac{3}{8}\n_{F'}\quad&&\textnormal{ if $F'\subset \partial \elem\setminus\partial\Om$ and ${\xi_p}\not\in\partial\Omega$},\\
\pm&\frac{1}{8}\n_{F'}\quad&&\textnormal{ otherwise}.\end{aligned}\right.
\end{equation}
We then have
\begin{align}
\normL[\mathsf{r}_0][\partial F][2] =& \sum_{{\xi_p}\in\partial F}\sum_{F'\in\mcal[F][{\xi_p}]}|\boldsymbol{\eta}_{F,F'}({\xi_p})\cdot\jump[v]({\xi_p})|^2\lesssim \sum_{{\xi_p}\in\partial F}\sum_{F'\in\mcal[F][{\xi_p}]}|\jump[v]({\xi_p})|^2\\
\lesssim& \sum_{{\xi_p}\in\partial F}\sum_{F'\in\mcal[F][{\xi_p}]}\frac{p^2}{h_{F'}}\normL[\jump[v]][F'][2]\lesssim \sum_{F': F\cap F'\neq \emptyset}\frac{p^2}{h_{F'}}\normL[\jump[v]][F'][2],
\end{align}
where the second step follows by the trace inequality \eqref{trace}. Finally we obtain
\begin{align}
\label{r0final}
\normL[\mathsf{r}_0][\elem][2]\lesssim& \frac{h_\elem}{p^2}\sum_{F\subset\partial \elem}\frac{h_F}{p^2}\sum_{F': F\cap F'\neq \emptyset}\frac{p^2}{h_{F'}}\normL[\jump[v]][F'][2]\\
\lesssim& \frac{h_\elem}{p^2}\sum_{F\in \faceElem}\normL[\jump[v]][F][2].
\end{align}
By combining \eqref{psifinal} and \eqref{r0final} the desired result follows. Finally, for $d=3$, inequality \eqref{wEst} reduces to
\begin{equation}
\label{wd3}
\normL[w][\elem][2]\lesssim \sum_{F\subset \partial \elem} \frac{h_\elem}{p^2}\normL[\psi_F][F][2]+\normL[\mathsf{r}_0][\elem][2]+\normL[\mathsf{r}_1][\elem][2].
\end{equation}
The first two terms on the right hand side can be bounded reasoning as before. To estimate the last term on the right hand side, we first observe that $\mathsf{r}_1\in\mathbb{Q}^{p}_0(\elem)$ and $\mathsf{r}_1\in\mathbb{Q}^{p}_0(F)$, and therefore we can apply again \eqref{invtrace} twice and obtain
\begin{equation}
\label{r1init}
\normL[\mathsf{r}_1][\elem][2]\lesssim \frac{h_\elem}{p^2}\sum_{F\subset\partial \elem}\normL[\mathsf{r}_1][F][2]\lesssim \frac{h_\elem}{p^2}\sum_{F\subset\partial \elem}\frac{h_F}{p^2}\sum_{E\subset\partial F}\normL[\mathsf{r}_1][E][2].
\end{equation}
In analogy to the estimate regarding $\mathsf{r}_0$, cf. \eqref{r0xi}, the following result can be proved
\begin{equation}
\mathsf{r}_1|_E = \sum_{{\xi_p}\in\mcal[V][1,E]}  w({\xi_p})\phi_{\xi_p} = \sum_{{\xi_p}\in\mcal[V][1,E]}\sum_{F'\in\mcal[F][{\xi_p}]} \boldsymbol{\eta}_{F,F'}({\xi_p})\jump[v]({\xi_p})\phi_{\xi_p},
\end{equation}
being \mcal[V][1,E] the set of interior nodes of the edge $E\subset\partial F$ and $\mcal[F][E] := \{F\in \face: E\subset F\}$. We then write
\begin{align}
\normL[\mathsf{r}_1][E][2]&\lesssim \sum_{F'\in\mcal[F][E]}\Big\|\sum_{{\xi_p}\in\mathcal{V}_{1,E}}\jump[v]({\xi_p})\phi_{\xi_p}\Big\|_{L^2(E)}^2\\
&\lesssim \sum_{F'\in\mathcal{F}_E} \normL[\jump[v]][E][2]+\Big\|\sum_{{\xi_p}\in\mcal[V][1,E]}\jump[v]({\xi_p})\phi_{\xi_p}- \jump[v]\Big\|_{L^2(E)}^2,
\end{align}
and observe that 
\begin{align}
\sum_{{\xi_p}\in\mcal[V][1,E]}\jump[v]({\xi_p})\phi_{\xi_p}- \jump[v]\in\mathbb{Q}^{p}_0(E),\quad\sum_{{\xi_p}\in\mcal[V][1,E]}\jump[v]({\xi_p})\phi_{\xi_p}=0 \textnormal{ on }\partial E,
\end{align}
which implies by \eqref{invtrace} and \eqref{trace}
\begin{equation}
\label{r1diff}
\Big\|\sum_{{\xi_p}\in\mcal[V][1,E]}\jump[v]({\xi_p})\phi_{\xi_p}- \jump[v]\Big\|_{L^2(E)}^2\lesssim\frac{h_E}{p^2}\sum_{{\xi_p}\in\partial E}|\jump[v]({\xi_p})|^2 \lesssim \normL[\jump[v]][E][2],
\end{equation}
hence,
\begin{equation}
\label{r1mid}
\normL[\mathsf{r}_1][E][2]\lesssim \sum_{F'\in\mcal[F][E]} \normL[\jump[v]][E][2].
\end{equation}
From \eqref{r1init}, \eqref{r1mid} and \eqref{trace}, we finally obtain
\begin{align}
\label{r1final}
\normL[\mathsf{r}_1][\elem][2]\lesssim& \frac{h_\elem}{p^2}\sum_{F\subset\partial \elem}\frac{h_F}{p^2}\sum_{E\subset\partial F}\normL[\mathsf{r}_1][E][2]\\
\lesssim& \frac{h_\elem}{p^2}\sum_{F\subset\partial \elem}\frac{h_F}{p^2}\sum_{E\subset\partial F} \sum_{F'\in\mcal[F][E]} \normL[\jump[v]][E][2]\\
\lesssim& \frac{h_\elem}{p^2}\sum_{F\in\faceElem} \normL[\jump[v]][F][2].
\end{align}
which combined with the analogous result for $\mathsf{r}_0$ and the bound on the norm of $\psi_F$, gives the thesis.
}
\section*{Acknowledgements}
Part of this work was developed during the visit of the second author
at the Pennsylvania State University. Special thanks go to the Center
for Computational Mathematics and Applications (CCMA) at the
Mathematics Department, Penn State for the hospitality and support.
The work of the fourth author was supported in part by NSF
DMS-1217142, NSF DMS-1418843, and Lawrence Livermore National
Laboratory through subcontract B603526.

\bibliographystyle{siam}
\bibliography{biblio}

\end{document}